\documentclass[12pt]{article}
\usepackage{amsmath,amssymb}
\newtheorem{theorem}{Theorem}
\newtheorem{unnumber}{}

\newtheorem{prepf}[unnumber]{Proof}
\newenvironment{proof}{\prepf\rm}{\endprepf}
\newcommand{\qed}{\qquad\Box}

\begin{document}
\title{Counting spanning trees containing a forest: a short proof}
\author{Peter J. Cameron and Michael Kagan}
\date{}
\maketitle
\begin{abstract}
Given a spanning forest $F$ of the complete graph on $n$ vertices, with
components of sizes $q_1,q_2,\ldots,q_m$, the number of spanning trees
containing $F$ is $q_1q_2\cdots q_mn^{m-2}$. We give a short self-contained
proof of this result.
\end{abstract}

By Cayley's Theorem, the number of spanning trees of the complete graph
$K_n$ is $n^{n-2}$.

A simple double counting argument shows that the number of spanning trees
containing an edge $e$ is $2n^{n-3}$.  (Count pairs $(T,e)$ where $T$ is a
spanning tree and $e$ an edge of $T$. Choosing $T$ first, the number of pairs
is $(n-1)n^{n-2}$, since $T$ has $n-1$ edges. The other way, there are
$n(n-1)/2$ edges, and so there must be $2n^{n-3}$ trees containing a given
edge.)

Now it is clear from symmetry that the number of spanning trees containing
two edges $e$ and $f$ depends only on whether $e$ and $f$ intersect. But
this cannot be found by such a simple argument.

We give a short proof of the following theorem of Moon~\cite{Moon},
showing in particular that the numbers mentioned are $3n^{n-4}$ if $e$ and
$f$ intersect, and $4n^{n-4}$ if they do not.

\begin{theorem}
Let $F$ be a spanning forest of the complete graph $K_n$, whose connected
components have $q_1,q_2,\ldots,q_m$ vertices, with
$q_1+q_2+\cdots+q_m=n$. Then the number of spanning trees containing $F$ is
$(q_1q_2\cdots q_m)n^{m-2}$.
\end{theorem}

\begin{proof}
We use the result of Kirchhoff~\cite{Kirchhoff}, asserting that the number of
spanning trees of a graph is equal to the cofactor of the Laplacian matrix
of the graph. (The graph is permitted to have multiple edges but not loops.
The Laplacian matrix has $(i,j)$ entry the negative of the number of edges from
the $i$th to the $j$th vertex if $i\ne j$, and the valency of the $i$th vertex
if $i=j$. It has row and column sums zero, and all its cofactors have the same
value.)

We first observe that the number of spanning trees containing $F$ is equal
to the number of spanning trees of the multigraph $M$ which has $m$ vertices
$v_1,\ldots,v_m$, where the number of edges from $v_i$ to $v_j$ is $q_iq_j$.
For if we take the given graph $K_n$, we can shrink each component of $F$ to
a single vertex to give the multigraph $M$; if $T$ is a spanning tree of $K_n$
containing $F$, then under this shrinking $T\setminus F$ becomes a spanning
tree of $M$, and every spanning tree of $M$ arises in this way.

So we have to count spanning trees of $M$. The Laplacian matrix of $M$ is
\[
\begin{pmatrix}
nq_1-q_1^2&-q_1q_2&-q_1q_3&\ldots&-q_1q_m\\
-q_2q_1&nq_2-q_2^2&-q_2q_3&\ldots&-q_2q_m\\
-q_3q_1&-q_3q_2&nq_3-q_3^2&\ldots&-q_3q_m\\
\vdots&\vdots&\vdots&\ddots&\vdots\\
-q_mq_1&-q_mq_2&-q_mq_3&\ldots&nq_m-q_m^2
\end{pmatrix}.
\]

We calculate the $(1,1)$ cofactor $L'$, the determinant obtained by deleting
the first row and column of this matrix. Take a factor $q_i$ from the $i$th
row, where the rows are numbrered $2$ to $m$. The cofactor reads
\[L'=q_2q_3\cdots q_m\left|
\begin{matrix}
n-q_2&-q_3&\ldots&-q_m\\
-q_2&n-q_3&\ldots&-q_m\\
\vdots&\vdots&\ddots&\vdots\\
-q_2&-q_3&\ldots&n-q_m
\end{matrix}\right|.\]

Now add all other columns to the leftmost one, and take out a factor $q_1$:
\[L'=q_1q_2q_3\cdots q_m\left|
\begin{matrix}
1&-q_3&\ldots&-q_m\\
1&n-q_3&\ldots&-q_m\\
\vdots&\vdots&\ddots&\vdots\\
1&-q_3&\ldots&n-q_m
\end{matrix}
\right|.\]

Finally, subtract the top row from each of the others:
\begin{eqnarray*}
L'&=&q_1q_2\cdots q_m\left|
\begin{matrix}
1&-q_3&\ldots&-q_m\\
0&n&\ldots&0\\
\vdots&\vdots&\ddots&\vdots\\
0&0&\ldots&n
\end{matrix}
\right|\\
&=&q_1q_2\cdots q_m n^{m-2}.\qed
\end{eqnarray*}

\end{proof}
\section*{Acknowledgements}
The authors are grateful to Victor Reiner for the idea of replacing the original graph with a multi-graph $M$, which allowed to shorten the paper dramatically.

\end{document}